\documentclass{amsart}

\usepackage{lineno}

\usepackage{amsmath,amssymb,amscd,amsthm}
\usepackage{epsfig,graphicx}
\usepackage{graphicx, amsmath, amsthm, amssymb, mathrsfs, amscd}
\usepackage[all]{xypic}
\xyoption{dvips}
\usepackage{hyperref}
\usepackage{bbm}
\usepackage{paralist}
\usepackage{verbatim}
\usepackage{todonotes}

\theoremstyle{definition}
\newtheorem{ntn}{Notation}[section]
\newtheorem{dfn}[ntn]{Definition} 
\theoremstyle{plain}
\newtheorem{lem}[ntn]{Lemma}
\newtheorem{prp}[ntn]{Proposition}
\newtheorem{thm}[ntn]{Theorem}
\newtheorem{cor}[ntn]{Corollary}

\theoremstyle{remark}
\newtheorem{rmk}[ntn]{Remark}
\newtheorem{exa}[ntn]{Example}
\numberwithin{equation}{section}

\newcommand{\ideal}[1]{{\left\langle#1\right\rangle}}
\newcommand{\into}{\hookrightarrow}

\newcommand{\p}{\partial}
\newcommand{\xymat}{\SelectTips{cm}{}\xymatrix}
\newcommand{\wt}{\widetilde}

\newcommand{\dd}{\mathbf{d}}
\newcommand{\A}{\mathcal A}
\newcommand{\C}{\mathcal C}
\newcommand{\F}{\mathcal F}
\renewcommand{\O}{\mathcal O}

\newcommand{\PP}{\mathbb P}
\newcommand{\CC}{\mathbb C}
\newcommand{\QQ}{\mathbb Q}
\newcommand{\ZZ}{\mathbb Z}

\newcommand{\sExt}{\mathscr Ext}

\newcommand{\Macaulay}{\texttt{Macaulay2}}

\DeclareMathOperator{\length}{length}
\DeclareMathOperator{\Ext}{Ext}

\DeclareMathOperator{\rank}{rank}
\DeclareMathOperator{\pdim}{pdim}

\begin{document}

\title[Chern classes for non-free arrangements]{Chern classes of logarithmic derivations for some non-free arrangements}

\author{Ngoc Anh Pham}
\address{
N. A. Pham\\
University of Kaiserslautern\\
Department of Mathematics\\
P.O.~Box 3049\\
67653 Kaiserslautern\\
Germany}
\email{pham@mathematik.uni-kl.de}
\thanks{The financial support by DAAD through the Ph.D.~program 
``Mathematics in Industry and Commerce'' at TU Kaiserslautern is gratefully acknowledged.}

\thanks{The research leading to these results has received funding from the People Programme (Marie Curie Actions) of the European Union's Seventh Framework Programme (FP7/2007-2013) under REA grant agreement n\textsuperscript{o} PCIG12-GA-2012-334355.}

\begin{abstract} 
Paolo Aluffi showed that the Chern--Schwartz--MacPherson class of the complement of a free arrangement agrees with the total Chern class of the sheaf of logarithmic derivations along the arrangement.
We describe the defect of equality of the two classes for locally tame arrangements with isolated non-free singular loci.
\end{abstract}



\date{\today}

\maketitle

\section{Introduction} 

Let $X$ be a non-singular variety over $\CC$ and $D$ be a hypersurface in $X$. 
In case $D$ is simple normal crossing, Paolo Aluffi first observed a coincidence of the Chern--Schwartz--MacPherson class $c_{SM}(X\setminus D)$ of the hypersurface complement $X\setminus D$ and the total Chern class $c(\Omega^1(\log D)^\vee)\cap [X]$ of the sheaf of differential $1$-forms with logarithmic poles along $D$ (see \cite[\S2]{Alu99}).
Based on this evidence, he raised the question whether the coincidence could persist for any free divisor $D$.
For projective arrangements whose affine cone is a free divisor, he gave a positive answer (see \cite[Thm.~4.1]{Alu13}).
In a sequence of articles, Aluffi's student Xia Liao first showed that the equality does not hold without further hypotheses (see \cite[Cor.~3.2]{Lia12}); then he proved the 
equality in question for all free divisors $D$ whose Jacobian ideal of linear type (see \cite{Lia12b}).
This latter algebraic condition holds true for locally quasihomogeneous hypersurfaces (see \cite[Thm.~5.6]{CN02}), so in particular for free arrangements.

In this note, we specialize to the class of non-free projective arrangements considered in \cite{DS12}.
Fix an $\ell$-dimensional vector space $V=\CC^\ell$. 
Let $\PP V=\PP^{\ell-1}$ be the corresponding projective space.
Consider a central arrangement of hyperplanes $\A$ in $V$ and let $\PP\A$ be the corresponding projective arrangement in $\PP V$.
By abuse of notation, we shall use $\A$ and $\PP\A$ as shorthands for $\bigcup_{H\in \A}H$ and $\bigcup_{H\in \A}\PP H$ respectively.

We denote by $L(\A)$ the intersection lattice of $\A,$ and by $L_c(\A)\subseteq L(\A)$ the sublattice of codimension-$c$ flats. For $X\in L(\A)$, let $\A_X=\{H\in\A\mid X\subseteq H\}$ be the localization of $\A$ at $X.$

By $M(\A)$ and $M(\PP\A)$ we denote the complements of $\A$ in $V$ and of $\PP\A$ in $\PP V$ respectively. 
Now the equality in question reads
\begin{equation}\label{for:1}
c(\Omega^1(\PP\A)^\vee)\cap [\PP V] = c_{SM}(M(\PP\A)).
\end{equation}
By Liao's result mentioned above, \eqref{for:1} holds true for locally free arrangements. 
In turn, our main result shows that \eqref{for:1} can fail in the non-free case, even if local quasihomogenity is imposed. 

\begin{thm}\label{thm:main}
Let $\PP\A$ be a locally tame arrangement in $\PP V$ with zero-dimensional non-free locus. Then 
\[
c(\Omega^1(\PP\A)^\vee)\cap [\PP V] = c_{SM}(M(\PP\A)) + ((-1)^{\ell-1}+(-1)^{\ell-2}(\ell-2)!)N(\PP\A)h^{\ell-1}
\]
where  $N(\PP\A)$ is the correction term from \cite[Def.~5.10]{DS12}, and $h$ denotes the class of a hyperplane.
\end{thm}

Besides recalling the relevant terminology for projective arrangements and Chern classes, we prove Theorem~\ref{thm:main} in \S\ref{section3}.

\subsection*{Acknowledgements}
I would like to thank Mathias Schulze for many invaluable suggestions while working on this paper.

\section{Preliminaries}\label{section1}

\subsection{Poincar\'e polynomials}

We use the notions in the introduction.
The \emph{Poincar\'e polynomial} of any (not necessarily central) arrangement $\A$ is defined as
\[
\pi(\A,t) = \sum_{X\in L(\A)}\mu(X)(-t)^{\rank(X)},
\]
where $\rank(X)$ is the codimension of $X$ in $V,$ and $\mu$ is the M\"obius function on the intersection poset $L(\A)$.
This terminology is justified by a classical topological interpretation due to Orlik and Solomon (see \cite[Thm.~5.93]{OT92}) in case of complex arrangements. 

\begin{thm}[Orlik--Solomon]\pushQED{\qed}
For a complex arrangement $\A$,
\[
\pi(\A,t) = \sum_{i = 0}^\ell\rank H^i(M(\A),\ZZ)t^i.\qedhere
\]
\end{thm}

Aluffi \cite[\S2.2]{Alu13} gave a direct proof of this theorem by computing the Grothendieck class of $M(\PP\A)$ and using the purity of its mixed Hodge structure. 

For $H=\{\alpha_H=0\}\in\A$, the deconing of a central arrangement $\A$ with respect to $H$ is the (possibly non-central) arrangement
\[
\dd\A=\bigcup_{H\ne H'\in\A}{H'\cap\{\alpha_H=1\}}
\]
in $\{\alpha_H=1\}\cong\CC^{\ell-1}$.
By \cite[Prop.~2.51]{OT92}, 
\[
\pi(\A,t) = (1+t)\pi(\dd\A,t).
\]
On the other hand, $M(\A)\cong M(\PP\A)\times\CC^*$ and $M(\PP\A)\cong M(\dd\A)$ as topological spaces.
We therefore call
\begin{equation}\label{eq:redPi}
\pi(\PP\A,t)= \sum_{i = 0}^{\ell-1}\rank H^i(M(\PP\A),\QQ)t^i = \pi(\dd\A,t) = \frac{\pi(\A,t)}{1+t} 
\end{equation}
the Poincar\'e polynomial of $\PP\A$.

\subsection{Logarithmic forms}

For an effective divisor $D$ in a smooth complex variety $X$, the sheaf $\Omega^p(\log D)$ of \emph{logarithmic differential $p$-forms along $D$} consists of rational differential $p$-forms $\omega$ on $X$ such that both $\omega$ and $d\omega$ have at most a simple pole along $D$ (see \cite{Saito80}). 
We consider these sheaves in case $D=\PP\A$ and $D=\A$ and denote them by $\Omega^p(\PP\A)$ and $\Omega^p(\A)$ respectively.
The sheaf $\Omega^p(\A)$ on the affine space $V$ can be identified with its module of global sections
\begin{equation}\label{eq:defOmegalog}
\Omega^p(\A) = \left\{\omega\in \frac{1}{f}\Omega_V^p \mid fd\omega\in \Omega_V^{p+1}\right\} = \left\{\omega\in \frac{1}{f}\Omega_V^p \mid df\wedge\omega\in \Omega_V^{p+1}\right\},
\end{equation}
where $\Omega_V^p=\bigwedge^p(\bigoplus_{i=1}^\ell Sdz_i)$ is the module of differential $p$-forms, and $f\in S:=\CC[z_1,\dots,z_\ell]$ is a homogeneous defining polynomial of $\A$  of degree $d:=\deg(f)$. 

The ring $S$ and hence the $S$-module $\Omega_V^p$ are naturally graded by setting $\deg z_i=\deg dz_i=1$.
Then also $\Omega^p(\A)$ is a graded $S$-module: An element $\omega\in\Omega^p(\A)$ is homogeneous of degree $i$ if $f\omega$ is homogeneous of degree $d+i$.

Denote by $\chi:=\sum_{i=1}^\ell z_i\p_{z_i}$ the Euler derivation.
From the proof of \cite[Thm.~8.13]{Har77} one deduces that the Euler sequence fits into a commutative diagram (see \eqref{eq:Euler} below)
\begin{equation}\label{eq:Euler}
\xymat{
0\ar[r] & \Omega^2_{\PP V}\ar[r]^{\varphi^2} & \wt{\Omega^2_V}\\
0\ar[r] & \Omega^1_{\PP V}\ar[u]^d\ar[r]^{\varphi^1} & \wt{\Omega^1_V}\ar[u]^{\wt d}\ar[r]^-{\ideal{\chi,-}} & \O_{\PP V}\ar[r] & 0\\
&\O_{\PP V}\ar[u]^-d\ar@{=}[r] & \wt S\ar[u]^-{\wt d}
}
\end{equation}
of locally free $\O_{\PP V}$-modules with exact rows.
Note that the map $\ideal{\chi,-}$ is obtained by applying $\wt-$ to the corresponding map $\Omega^1_V\to S$.
Since $\wt{\Omega^1_V}\cong\O_{\PP V}(-1)^{\ell}$ it does not split. 
However such a splitting holds true after passing to the logarithmic analogue.

\begin{prp}\label{prp:logEuler}
There is a split exact logarithmic Euler sequence  
\begin{equation}\label{eq:logEuler}
\xymat{
0\ar[r] & \Omega^1(\PP\A)\ar[r] & \wt{\Omega^1(\A)}\ar[r]^-{\ideal{\chi,-}}\ar[r] & \O_{\PP V}\ar[r] & 0.
}
\end{equation}
\end{prp}

\begin{proof}
Tensoring the middle row of \eqref{eq:Euler} by $\O_{\PP V}(*\PP\A)$, yields a diagram
\begin{equation}
\xymat{
0\ar[r] & \Omega_{\PP V}^1\ar[r]^{\varphi^1}\ar@{^(->}[d] & \wt{\Omega_V^1}\ar[r]^{\ideal{\chi,-}}\ar[r]\ar@{^(->}[d] & \O_{\PP V}\ar[r]\ar@{=}[d] & 0\\
0\ar[r] & \Omega^1(\PP\A)\ar@{^(->}[d]\ar@{-->}[r]^{\varphi^1}\ar[dr] & \wt{\Omega^1(\A)}\ar@{^(->}[d]\ar[r]^{\ideal{\chi,-}} & \O_{\PP V}\ar@{^(->}[d]\ar[r] & 0\\
0\ar[r] & \Omega^1_{\PP V}(*\PP\A)\ar[r] & \wt{\Omega^1_V(*\A)}\ar[r]^-{\ideal{\chi,-}} & \O_{\PP V}(*\PP\A)\ar[r] & 0
}
\end{equation}
with exact upper and lower row.
The extension of $\ideal{\chi,-}$ obtained by applying $\wt-$ to the corresponding map
\begin{equation}\label{eq:Eulercontract}
\ideal{\chi,-}:\Omega^1_V(\A)\to S.
\end{equation}
The dashed map is injective once it exists and existence can be checked locally.
Leaving this part to the reader, we prove exactness of \eqref{eq:logEuler} in the middle.
It suffices to look at the chart $U_\ell=D(z_\ell)$ of $\PP V$ with coordinate ring $R:=\CC[x_1,\dots,x_{\ell-1}]$ where $x_i=\frac{z_i}{z_\ell}$, $i=1,\dots,\ell-1$.
By \cite[Thm.~8.13]{Har77}, the map 
\[
\xymat{
\bigoplus_{i=1}^{\ell-1}Rdx_i=\Omega^1_{\PP V}(U_\ell)\ar[r]^-{\varphi^1_{U_\ell}} & \wt{\Omega^1_V}(U_\ell)=\Omega^1_{V,(z_\ell)}=\bigoplus_{i=1}^{\ell}R\frac{dz_i}{z_\ell}
}
\]
is given explicitly by
\begin{equation}\label{eq:Eulerd}
\varphi^1_{U_\ell}(dx_i)=d\left(\frac{z_i}{z_\ell}\right)=\frac{dz_i}{z_\ell}-x_i\frac{dz_\ell}{z_\ell}.
\end{equation}
Let $\omega\in\wt{\Omega^1(\A)}(U_\ell)$ such that $\ideal{\chi,\omega}=0$.
By \eqref{eq:defOmegalog} and \eqref{eq:Euler}, we can write
\begin{equation}\label{eq:omega}
\omega=\frac{\eta}{f(z)}=\frac{\eta}{z_\ell^d}\frac{z_\ell^d}{f(z)}=\frac{\varphi^1_{U_\ell}(\eta')}{f(x)}=\varphi^1_{U_\ell}\left(\frac{\eta'}{f(x)}\right)
\end{equation}
where $\eta\in\Omega_{V,z_\ell}^1$, $\deg(\eta)=\deg(f)=d$, and $\eta'\in\Omega^1_{\PP V}(U_\ell)$.
Note that \eqref{eq:defOmegalog} is compatible with localization and multiplying $f$ by a unit.
In particular, the definition of $\Omega^1(\A)_{z_\ell}$ is invariant under replacing $f(z)$ by $f(x)=\frac{f(z)}{z_\ell^d}$. 
Thus, by \eqref{eq:defOmegalog} and \eqref{eq:Euler}, we have
\[
\varphi^2_{U_\ell}(df(x)\wedge\eta')=\wt df(x)\wedge\varphi^1_{U_\ell}(\eta')\in f(x)\wt{\Omega^2_V}(U_\ell).
\]
Due to \eqref{eq:Eulerd} $\varphi^2_{U_\ell}$ is an inclusion of free $\O_{\PP V}(U_\ell)$-modules, and hence
\[
df(x)\wedge\eta'\in f(x)\Omega^2_{\PP V}(U_\ell).
\]
This implies $\frac{\eta'}{f(x)}\in\Omega^1(\PP\A)(U_\ell)$ by definition. 
The exactness of \eqref{eq:logEuler} in the middle follows from \eqref{eq:omega}.
Finally the splitting of \eqref{eq:logEuler} is induced by the the section $\frac1d\frac{df}f$ of \eqref{eq:Eulercontract}.
\end{proof}

\begin{rmk}
The kernel $\Omega^1_0(\A)$ of \eqref{eq:Eulercontract} is the module of \emph{relative logarithmic differential $1$-forms along $\A$}.
From Proposition~\ref{prp:logEuler}, we obtain (see also \cite[Prop.~2.12]{DS12})
\begin{equation}\label{eq: Omega zero}
\Omega^1(\PP\A)\cong\wt{\Omega_0^1(\A)}.
\end{equation}
With $\Omega^1(\A)$ also its direct summand $\Omega^1_0(\A)$ is reflexive module.
In particular, both modules have no $\ideal{z}$-torsion and hence $\Omega^1(\A)=\Gamma^*\wt{\Omega^1(\A)}$ as well as $\Omega^1_0(\A)=\Gamma^*\wt{\Omega^1_0(\A)}$.
\end{rmk}

By the splitting of \eqref{eq:logEuler},
\begin{equation}\label{eq: Omega split}
\wt{\Omega^1(\A)}(1)\cong\Omega^1(\PP\A)(1)\oplus\O_{\PP V}(1),
\end{equation}
which yields the following relation of Chern polynomials:
\begin{equation}\label{eq:redCh}
c_t(\wt{\Omega^1(\A)}(1)) = (1+t) c_t(\Omega^1(\PP\A)(1)).
\end{equation}

We call the arrangement $\PP\A$ \emph{locally free} if the sheaf $\Omega^1(\PP\A)$ is a vector bundle. 
By \eqref{eq:logEuler}, this is equivalent to the local freeness of the central arrangement $\A$ in the
sense of \cite{Ser93} (see \cite[Lem.~2.16]{DS12}). 
For locally free arrangements, Musta{\c{t}}{\v{a}} and Schenck~\cite[Thm.~4.1]{MS01} established a relation between the Poincar\'e polynomial and the Chern polynomial of logarithmic forms.
Using \eqref{eq:redPi} and \eqref{eq:redCh}, it can be formulated as follows (see \cite[Thm.~5.9]{DS12}):

\begin{thm}[Musta{\c{t}}{\v{a}}--Schenck]
If $\PP\A$ is a locally free arrangement in $\PP V,$ then 
\begin{equation}\label{proj-Mustata}
c_t(\Omega^1(\PP\A)(1)) = \pi(\PP\A,t).
\end{equation}
\end{thm}

\begin{proof}
We only need to show that the hypothesis of \cite{MS01} that $\A$ is an essential arrangement can be dropped.
To this end, let $\A'$ be a non-essential central arrangement. 
Proceeding by induction on the dimension of the center of $\A'$, we may assume that there is a hyperplane $H$ in $V$ such that (\ref{proj-Mustata}) holds true for $\A'':=\{H'\cap H\mid H'\in\A'\}$.
We then have 
\begin{equation}\label{piAess}
\pi(\PP\A'',t)=\pi(\PP\A',t)
\end{equation}
by \cite[Lem.~2.50]{OT92} and \eqref{eq:redPi}.

Let $i\colon\PP H\into\PP V$ denote the inclusion.
Then we have an exact sequence 
\[
\xymat{
0\ar[r] & N^\vee\ar[r] & i^*\Omega^1(\PP\A')\ar[r] & \Omega^1(\PP\A'')\ar[r] & 0,
}
\]
where $N=\O_{\PP H}(1)$ denotes the normal bundle to $\PP H$ in $\PP V$.
Using the projection formula (see \cite[Prop.~2.5.(c) and 2.6.(b)]{Ful98}), it follows that
\begin{align*}
i_*(c_t(\Omega^1(\PP\A'')(1))\cap[\PP H])
&=i_*(c_t(i^*\Omega^1(\PP\A')(1))\cap[\PP H])\\
&=c_t(\Omega^1(\PP\A')(1))\cap i_*i^*[\PP V]\\
&=c_t(\Omega^1(\PP\A')(1))\cap c_1(\O_{\PP V}(1))\cap [\PP V]\\
&=tc_t(\Omega^1(\PP\A')(1))\cap [\PP V].
\end{align*}
A similar argument can be found in \cite[\S4]{Lia12a}.
On the other hand, regarding $t = c_1(\O_{\PP H}(1))=c_1(i^*\O_{\PP V}(1))\in A^1(\PP H)$ and writing $\pi({\PP\A''},t) = \sum_kb_kt^k$ for some $b_k\in\ZZ$, we obtain
\begin{align*}
i_*( \pi({\PP\A''},t)\cap[\PP H])
&=i_*(\pi({\PP\A''},c_1(i^*\O_{\PP V}(1)))\cap [\PP H])\\
&=\sum_kb_k(i_*(c_1(i^*\O_{\PP V}(1)))^k\cap i_*i^*[\PP V])\\
&=\sum_kb_kc_1(\O_{\PP H}(1))^{k+1}\cap [\PP V]\\
&=t\pi({\PP\A''},t)=t\pi({\PP\A'},t),
\end{align*}
where the last equality is \eqref{piAess}.
Combining the above two equalities yields
\[
c_t(\Omega^1(\PP\A')(1))\cap [\PP V] =\pi({\PP\A'},t)
\]
as claimed.
\end{proof}

For the Chern--Schwartz--MacPherson class (see \S\ref{sec csm}), Aluffi~\cite[Proof of Thm.~4.1]{Alu13} proved a similar reduction to the essential case.

\subsection{Chern--Schwartz--MacPherson classes}\label{sec csm}

The Chern--Schwartz--MacPherson class of a possibly singular variety was constructed by MacPherson in \cite{MacPherson}. If the variety is non-singular,
this class is the total (homology) Chern class of the tangent bundle.

For the purpose of this paper, we consider the projective space $\PP V$ introduced in the previous section. 
A \emph{constructible function} on a projective scheme $X\subseteq \PP V$ is a formal $\sum_{W}n_W\cdot\mathbbm{1}_W,$
where the $W$ are closed subvarieties of $X$, almost all integers $n_W$ are zero, and the $\mathbbm{1}_W$ are characteristic functions. 
They form a group denoted by $\C(X)$. 
If $f\colon X\rightarrow Y$ is a proper morphism between projective schemes, the push-forward of constructible functions $\C(f): \C(X)\rightarrow \C(Y)$ is defined by
\[
f_*(\mathbbm{1}_W)(p) = \chi(f^{-1}(p)\cap W) , \text{ for } p\in Y.
\]
 Here, $\chi$ is the topological Euler characteristic of a subset of $X$ in the complex topology. This definition makes $\C$ into a covariant functor from the category of projective schemes to the category of abelian groups. 
The functor $\C$ of constructible functions above and  the \emph{Chow group} functor $A_*$ (see \cite[Ch.~1]{Ful98}) have a nice relation. 
Grothendieck -- Deligne conjectured and MacPherson proved that there is a natural transformation
$c_*\colon\C\rightarrow A_*$,
such that if $X$ is non-singular then 
\[
c_*(\mathbbm{1}_X) = c(T_X)\cap [X],
\]
where $c(T_X)$ is the total Chern class of the tangent bundle on $X$ (see \cite{MacPherson}).
The \emph{(total) Chern--Schwartz--MacPherson (CSM) class} of a (possibly singular) subvariety $X\subseteq \PP V$ is defined as follows:
\[
c_{SM}(X) := c_*(\mathbbm{1}_X)\in A_*(X),
\]
(see in \cite[\S1.3]{Alu05}).
We write the Chow group of $\PP V$ as $A_*(\PP V) \cong\ZZ[h]/\langle h^\ell\rangle$, where $h$ denotes the class of a hyperplane in $\PP V$.
Identifying $c_{SM}(X)$ with its image under the canonical map $A_*(X)\into A_*(\PP V)$, we have
\[
c_{SM}(X) = a_0 + a_1h + \cdots + a_{\ell-1}h^{\ell-1}.
\]

\subsection{Aluffi's formula}\label{sec Aluffi}

For a projective arrangement $\PP\A$, Aluffi~\cite[Cor.~3.2]{Alu13} proved the following relation of the CSM class $c_{SM}(M(\PP\A))$ of the complement $M(\PP\A)$ and the Poincar\'e polynomial $\pi(\PP\A,t)$.

\begin{lem}[Aluffi]
For any projective arrangement $\PP\A\subset \PP V,$ one has
\begin{equation}\label{c-p}
c_{SM}(M(\PP\A)) = (1+h)^{\ell-1}\pi\biggl(\PP\A, \frac{-h}{1+h}\biggl)\cap [\PP V].
\end{equation}
\end{lem}

Formula (\ref{c-p}) holds true for any projective arrangement. 
On the other hand, if $\PP\A$ is locally free, $\Omega^1(\PP\A)$ is a vector bundle of rank $\ell-1$ on $\PP V$ whose total Chern class is directly related to the Poincar\'e polynomial due to (\ref{proj-Mustata}).
Combining formulas (\ref{c-p}) and (\ref{proj-Mustata}) yields a proof of the following result of Aluffi~
\cite[Thm~4.1]{Alu13}.

\begin{thm}[Aluffi]\label{theo-Aluffi}
Suppose that $\PP\A$ is a locally free arrangement in $\PP V.$ Then
\[
c(\Omega^1(\PP\A)^\vee)\cap [\PP V] = c_{SM}(M(\PP\A)).
\]
\end{thm}

\section{Defect for zero-dimensional non-free locus}\label{section3}

As mentioned above, Aluffi's Theorem $\ref{theo-Aluffi}$ can be derived from formula $\eqref{proj-Mustata}$.
The latter has been extended to the case of locally tame arrangements in \cite[Thm.~5.13]{DS12} by adding a correction term.
Using this formula we shall extend Aluffi's Theorem to locally tame arrangements.

We first recall the relevant terminology and results from \cite{DS12}.
Relaxing the condition of local freeness at a finite set of points, the following integer measures the failure of local freeness.

\begin{dfn}[{\cite[Def.~5.10]{DS12}}]\label{def N}\
\begin{asparaenum}
\item If $\A$ is an affine arrangement in $V$ with zero-dimensional non-free locus, we set 
\[
N(\A)=\length(\Ext^1_S(\Omega^1(\A),S)).
\]
\item If $\PP\A$ is a projective arrangement in $\PP V$ with zero-dimensional non-free locus, we set
\[
N(\PP\A) = h^{0}(\sExt^{1}_{\O_{\PP V}}(\Omega^{1}(\PP\A),\O_{\PP V}))= \sum_{X\in L_{\ell-1}(\A) }N(\A_X)
\]
where $\A_X$ is considered as an arrangement in an affine chart of $\PP V$.
\end{asparaenum}
\end{dfn}

\begin{rmk}\label{rmk zero loc}\
\begin{asparaenum}

\item The sheaf $\sExt^{p}_{\O_{\PP V}}(\Omega^{1}(\PP\A),\O_{\PP V})$ in Definition~\ref{def N} is supported at the non-free locus of $\PP\A$ and vanishes for $p>1$.
Moreover, $N(\PP\A) = 0$ means exactly that $\PP\A$ is locally free (see \cite[Rem.~2.7.1]{Hart80}).

\item Any arrangement in $\PP^3$ has zero-dimensional non-free locus since $\Omega^{1}(\PP\A)$ is a reflexive sheaf (see \cite[Cor.~1.4]{Hart80}).

\end{asparaenum}
\end{rmk}

The following notion is a projective version of the classical tameness condition frequently considered in algebraic arrangement theory:
$\A$ is \emph{tame} if $\pdim\Omega^p(\A)\le p$ for all $0\leq p\leq \ell$.

\begin{dfn}[{\cite[Def.~2.13]{DS12}}]\label{tame} 
The projective arrangement $\PP\A$ is called \emph{locally tame} if $\Omega^p(\PP\A)$ has a locally free resolution of length $p,$ for $0\leq p \leq \ell - 1$.
\end{dfn}

\begin{rmk}\label{rmk tame}\
\begin{asparaenum}

\item If $\A$ is tame then $\PP\A$ is locally tame.

\item All arrangements in $\PP^3$ are locally tame.
This follows from reflexivity of $\Omega^1(\PP\A)$, \cite[Prop.~1.3]{Hart80}, and the Auslander--Buchsbaum formula.

\end{asparaenum}
\end{rmk}

Denham and Schulze \cite[Thm.~5.13]{DS12} proved the following

\begin{thm}[Denham--Schulze]\label{D-S}
If $\PP\A$ is a locally tame arrangement in $\PP V$ with zero-dimensional non-free locus, then
\begin{equation}\label{den}
c_t(\Omega^1(\PP\A)(1)) = \pi(\PP\A,t) + N(\PP\A) t^{\ell-1}.
\end{equation}
\end{thm}

The following can be found in \cite[Ex.~15.3.1]{Ful98}, which generalizes a result of Hartshorne \cite[Lem.~2.7]{Hart80}.

\begin{lem}\label{chern:point} 
The Chern polynomial of a reduced point in $\PP^d$ is $1 + (-1)^{d-1}(d-1)!t^d$.
\end{lem}

We can now prove our main result extending Theorem~$\ref{theo-Aluffi}$ to the case of locally tame arrangements.

\begin{proof}[Proof of Theorem~\ref{thm:main}]
Local tameness provides a locally free resolution 
\[
\xymat{
0\ar[r] & \F_1\ar[r] & \F_0\ar[r] & \Omega^1(\PP\A)(1)\ar[r] & 0.
}
\]
As $\PP\A$ has zero-dimensional non-free locus (see Remark~\ref{rmk zero loc}), dualizing leads to an exact sequence
\[
\xymat{
0\ar[r] & \Omega^1(\PP\A)^\vee(-1)\ar[r] & \F_0^\vee\ar[r] &\F_1^\vee\ar[r] & \sExt_{\O_{\PP V}}(\Omega^1(\PP\A), \O_{\PP V})\ar[r] & 0.
}
\]
Applying the Whitney sum formula (see \cite[Thm.~3.2.(e)]{Ful98}) gives
\begin{align*}
c_t(\Omega^1(\PP\A)^\vee(-1)) 
&= c_t(\F_0^\vee)c_t(\F_1^\vee)^{-1}c_t(\sExt^{1}_{\O_{\PP V}}(\Omega^1(\PP\A), \O_{\PP V}))\\
&=c_{-t}(\F_0)c_{-t}(\F_1)^{-1}c_t(\sExt^{1}_{\O_{\PP V}}(\Omega^1(\PP\A), \O_{\PP V}))\\
&= c_{-t}(\Omega^1(\PP\A)(1))(1 + (-1)^{\ell-2}(\ell-2)!N(\PP\A)t^{\ell-1})
\end{align*}
where we use \cite[Rem.~3.2.3(a)]{Ful98} and Lemma~\ref{chern:point} for the second and third equality.
Substituting $\eqref{den}$ yields an identity
\begin{align}\label{eq:cp}
c_t(\Omega^1(\PP\A)^\vee(-1))
&=(\pi(\PP\A,-t) + (-1)^{\ell-1}N(\PP\A)t^{\ell-1})\\
\nonumber&\ \ \cdot(1+(-1)^{\ell-2}(\ell-2)!N(\PP\A)t^{\ell-1})\\
\nonumber &=\pi(\PP\A,-t) + ((-1)^{\ell-1}+(-1)^{\ell-2}(\ell-2)!)N(\PP\A)t^{\ell-1}
\end{align}
in the Chow ring $A^*(\PP V)\cong \ZZ[t]/\langle t^{\ell}\rangle$.
Expand the Poincar\'e polynomial of $\PP\A$ as
\[
\pi(\PP\A, t) = \sum_{i = 0}^{\ell-1}b_it^i.
\]
Then (\ref{eq:cp}) implies
\[
c_i(\Omega^1(\PP\A)^\vee(-1)) = (-1)^ib_i, \quad\text{ for all } i = 0,\dots,\ell-2,
\]
and 
\[
c_{\ell-1}(\Omega^1(\PP\A)^\vee(-1)) = (-1)^{\ell-1}b_{\ell-1} + ((-1)^{\ell-1}+(-1)^{\ell-2}(\ell-2)!)N(\PP\A).
\]
Rewriting $\Omega^1(\PP\A)^\vee = (\Omega^1(\PP\A)^\vee(-1))(1)$ and applying the formula for Chern classes of the twisted sheaves (see \cite[Lem.~ 2.1]{Hart80}), we deduce
\begin{align}\label{eq:ckbi}
c_k(\Omega^1(\PP\A)^\vee)
& = c_k((\Omega^1(\PP\A)^\vee(-1))(1)) \\
\nonumber&= \sum_{i = 0}^k\binom{\ell -1-i}{k-i}c_{i}(\Omega^1(\PP\A)^\vee(-1))\\
\nonumber& =\sum_{i = 0}^k\binom{\ell -1-i}{k-i}(-1)^{i}b_{i}, \quad\text{ for all } k = 0,\dots,\ell-2,
\end{align}
and 
\[
c_{\ell-1}(\Omega^1(\PP\A)^\vee) = \sum_{i = 0}^{\ell-1}\binom{\ell -1-  i}{\ell-1-i}(-1)^{i}b_{i} + ((-1)^{\ell-1}+(-1)^{\ell-2}(\ell-2)!)N(\PP\A).
\]
Using $(\ref{c-p})$ this implies the claim:
\begin{align*}
c_{SM}(M(\PP\A)) &= (1+h)^{\ell-1}\pi(\PP\A, \frac{-h}{1+h})\\
& = \sum_{i = 0}^{\ell-1}(-1)^i b_ih^i(1+h)^{\ell-1-i}\\
& = \sum_{i = 0}^{\ell-1}(-1)^i b_i\sum_{i+j\le\ell-1}{\ell-1-i\choose j}h^{i+j}\\
& = \sum_{k=0}^{\ell-1}\sum_{i = 0}^k(-1)^i b_i{\ell-1-i\choose k-i}h^k\\
& = \sum_{k=0}^{\ell-1}c_k(\Omega^1(\PP\A)^\vee)h^k -((-1)^{\ell-1}+(-1)^{\ell-2}(\ell-2)!)N(\PP\A)h^{\ell-1}\\
& = c(\Omega^1(\PP\A)^\vee)\cap [\PP V]-((-1)^{\ell-1}+(-1)^{\ell-2}(\ell-2)!)N(\PP\A)h^{\ell-1}.
\end{align*}
\end{proof}

By Remarks~\ref{rmk zero loc} and \ref{rmk tame}, the hypotheses of Theorem~$\ref{thm:main}$ are automatically fulfilled by any arrangement in $\PP^3$.

\begin{cor}\label{cor:1}
If $\PP\A$ is an arrangement in $\PP^3,$ then
\[
c(\Omega^1(\PP\A)^\vee)\cap [\PP^3] = c_{SM}(M(\PP\A)) + N(\PP\A)h^3.
\] 
\end{cor}

The following example illustrate the formula in Corollary~\ref{cor:1}.

\begin{exa}
Consider the arrangement
\[
\PP\A =\{xyzw(x-w)(y-w)(x+y+z)(x-y+z)=0\}\subset \PP^3.
\]
Aluffi gave an algorithm for computing the CSM class of projective schemes (see \cite{Alu03}). 
Using the \Macaulay~implementation~\cite{M2} of this algorithm, we find
\[
c_{SM}(\PP\A) =  8h - h^2 + 9h^3.
\]
Hence,
\begin{align*}
c_{SM}(M(\PP\A)) &= c_{SM}(\PP^3) - c_{SM}(\PP\A)\\
&= ((1+h)^4-h^4)- (8h - h^2 + 9h^3)\\
&= 1 - 4h + 7h^2-5h^3.
\end{align*}
By another \Macaulay~calculation, $N(\PP\A) = 3$ and hence
\[
c(\Omega^1(\PP\A)^\vee)\cap [\PP^3] = 1 - 4h + 7h^2-2h^3
\]
by Corollary $\ref{cor:1}$.

The same result can be found following a different approach:
Using \eqref{eq: Omega zero} and \cite{Chan04}, the Chern polynomial $c_t(\Omega^1(\PP\A)^\vee)$ can be obtained 
by from the Hilbert polynomial $P(\Omega_0^1(\A)^\vee,t)$. 
Applying \Macaulay's \texttt{HyperplaneArrangements} package, we compute
\[
P(\Omega^1(\A)^\vee(-1), t)=\frac{2}{3}t^ 3 - \frac{5}{3}t + 2.
\]
Since $\Omega^1(\A)^\vee(-1)\cong \Omega_0^1(\A)^\vee(-1)\oplus S(-1)$ (see \eqref{eq: Omega split}), we thus get
\[
P(\Omega_0^1(\A)^\vee(-1),t) = \left(\frac{2}{3}t^3 - \frac{5}{3}t + 2\right) - \binom{t+2}{3} = \frac{1}{2}t^3 -\frac{1}{2}t^2 -2t + 2.
\]
A simple computation shows that
\[
P(\Omega_0^1(\A)^\vee, t) = \frac{1}{2}t^3+t^2-\frac{3}{2}t = 5P(\O_{\PP^3}(-2),t) - 2P(\O_{\PP^3}(-3),t),
\]
which gives
\[
c_t(\Omega^1(\PP\A)^\vee)=c_t(\wt{\Omega_0^1(\A)^\vee})=\frac{(1-2t)^5}{(1-3t)^2} = 1 - 4t + 7t^2-2t^3.
\]
Corollary $\ref{cor:1}$ now yields 
\[
c_{SM}(M(\PP\A)) = 1 - 4h + 7h^2-5h^3,
\]
as above.
\end{exa}

\bibliographystyle{amsalpha}
\bibliography{tarr}

\providecommand{\bysame}{\leavevmode\hbox to3em{\hrulefill}\thinspace}
\providecommand{\MR}{\relax\ifhmode\unskip\space\fi MR }
\providecommand{\MRhref}[2]{%
  \href{http://www.ams.org/mathscinet-getitem?mr=#1}{#2}
}
\providecommand{\href}[2]{#2}
\begin{thebibliography}{CMNM02}

\bibitem[Alu99]{Alu99}
Paolo Aluffi, \emph{Differential forms with logarithmic poles and
  {C}hern-{S}chwartz-{M}ac{P}herson classes of singular varieties}, C. R. Acad.
  Sci. Paris S\'er. I Math. \textbf{329} (1999), no.~7, 619--624. \MR{1717120
  (2001d:14008)}

\bibitem[Alu03]{Alu03}
\bysame, \emph{Computing characteristic classes of projective schemes}, J.
  Symbolic Comput. \textbf{35} (2003), no.~1, 3--19. \MR{1956868 (2004b:14007)}

\bibitem[Alu05]{Alu05}
\bysame, \emph{Characteristic classes of singular varieties}, Topics in
  cohomological studies of algebraic varieties, Trends Math., Birkh\"auser,
  Basel, 2005, pp.~1--32. \MR{2143071 (2006h:14010)}

\bibitem[Alu13]{Alu13}
\bysame, \emph{Grothendieck classes and {C}hern classes of hyperplane
  arrangements}, Int. Math. Res. Not. IMRN (2013), no.~8, 1873--1900.
  \MR{3047491}

\bibitem[Cha04]{Chan04}
C.-Y.~Jean Chan, \emph{A correspondence between {H}ilbert polynomials and
  {C}hern polynomials over projective spaces}, Illinois J. Math. \textbf{48}
  (2004), no.~2, 451--462. \MR{2085419 (2006g:14017)}

\bibitem[CMNM02]{CN02}
Francisco Calder{\'o}n-Moreno and Luis Narv{\'a}ez-Macarro, \emph{The module
  {$\mathscr{D}f\sp s$} for locally quasi-homogeneous free divisors},
  Compositio Math. \textbf{134} (2002), no.~1, 59--74. \MR{MR1931962
  (2003i:14016)}

\bibitem[DS12]{DS12}
Graham Denham and Mathias Schulze, \emph{Complexes, duality and {C}hern classes
  of logarithmic forms along hyperplane arrangements}, Arrangements of
  hyperplanes---{S}apporo 2009, Adv. Stud. Pure Math., vol.~62, Math. Soc.
  Japan, Tokyo, 2012, pp.~27--57. \MR{2933791}

\bibitem[Ful98]{Ful98}
William Fulton, \emph{Intersection theory}, second ed., Ergebnisse der
  Mathematik und ihrer Grenzgebiete. 3. Folge. A Series of Modern Surveys in
  Mathematics [Results in Mathematics and Related Areas. 3rd Series. A Series
  of Modern Surveys in Mathematics], vol.~2, Springer-Verlag, Berlin, 1998.
  \MR{1644323 (99d:14003)}

\bibitem[GS]{M2}
Daniel~R. Grayson and Michael~E. Stillman, \emph{Macaulay2, a software system
  for research in algebraic geometry}, Available at
  \href{http://www.math.uiuc.edu/Macaulay2}{http://www.math.uiuc.edu/Macaulay2}.

\bibitem[Har77]{Har77}
Robin Hartshorne, \emph{Algebraic geometry}, Springer-Verlag, New York, 1977,
  Graduate Texts in Mathematics, No. 52. \MR{MR0463157 (57 \#3116)}

\bibitem[Har80]{Hart80}
\bysame, \emph{Stable reflexive sheaves}, Math. Ann. \textbf{254} (1980),
  no.~2, 121--176. \MR{597077 (82b:14011)}

\bibitem[Lia12a]{Lia12b}
Xia Liao, \emph{Chern classes of logarithmic derivations for free divisors with
  jacobian ideal of linear type}, arXiv:1210.6079 (2012).

\bibitem[Lia12b]{Lia12}
\bysame, \emph{Chern classes of logarithmic vector fields}, J. Singul.
  \textbf{5} (2012), 109--114. \MR{2928937}

\bibitem[Lia12c]{Lia12a}
\bysame, \emph{Chern classes of logarithmic vector fields for
  locally-homogeneous free divisors}, arXiv:1210.6079 (2012).

\bibitem[Mac74]{MacPherson}
R.~D. MacPherson, \emph{Chern classes for singular algebraic varieties}, Ann.
  of Math. (2) \textbf{100} (1974), 423--432. \MR{0361141 (50 \#13587)}

\bibitem[MS01]{MS01}
Mircea Musta{\c{t}}{\v{a}} and Henry~K. Schenck, \emph{The module of
  logarithmic {$p$}-forms of a locally free arrangement}, J. Algebra
  \textbf{241} (2001), no.~2, 699--719. \MR{1843320 (2002c:32047)}

\bibitem[OT92]{OT92}
Peter Orlik and Hiroaki Terao, \emph{Arrangements of hyperplanes}, Grundlehren
  der Mathematischen Wissenschaften [Fundamental Principles of Mathematical
  Sciences], vol. 300, Springer-Verlag, Berlin, 1992. \MR{1217488 (94e:52014)}

\bibitem[Sai80]{Saito80}
Kyoji Saito, \emph{Theory of logarithmic differential forms and logarithmic
  vector fields}, J. Fac. Sci. Univ. Tokyo Sect. IA Math. \textbf{27} (1980),
  no.~2, 265--291. \MR{586450 (83h:32023)}

\bibitem[Yuz93]{Ser93}
Sergey Yuzvinsky, \emph{Free and locally free arrangements with a given
  intersection lattice}, Proc. Amer. Math. Soc. \textbf{118} (1993), no.~3,
  745--752. \MR{1160307 (93i:52022)}

\end{thebibliography}
\end{document}